\documentclass[11pt]{article}
\usepackage{a4wide}
\usepackage{amsmath}

\usepackage{amsfonts}
\usepackage{amssymb}
\usepackage{euscript}
\newenvironment{proof}{\noindent {\it Proof.~~}\ }{\  \rule{1mm}{2mm}\medskip}

\newenvironment{proofof}[2]{\noindent {\it Proof of #1}~#2: \
}{~\rule{1mm}{2mm}\medskip}
\newtheorem{theorem}{Theorem}
\newtheorem{lemma}[theorem]{Lemma}
\newtheorem{corollary}[theorem]{Corollary}
\newtheorem{proposition}[theorem]{Proposition}

\def\zp{{\mathbb Z}/p{\mathbb Z}}

\newcommand{\subgp}[1]{\langle{#1}\rangle}

\begin{document}
\title{ Some additive applications of the isoperimetric approach}

\author{ Yahya O. Hamidoune\thanks{UPMC Paris06,     {\tt hamidoune@math.jussieu.fr} }
}

\maketitle

\begin{abstract}
Let $G$ be a group and let $X$ be a finite subset.
 The isoperimetric method investigates  the objective function
$|(XB)\setminus X|$, defined on the subsets $X$ with $|X|\ge k$ and
$|G\setminus (XB)|\ge k$.

In this paper we  present all the basic facts about the isoperimetric
method. We improve some of our previous results and obtain
generalizations and short proofs for several known results. We also
give some new applications.

Some of the results obtained here will be used in  coming papers
\cite{hkneser,hkemp,hkempp} to improve Kempermann structure Theory.

\end{abstract}

\section{Introduction}

The starting point of set product estimation is the inequality
$|AB|\ge \min(|G|,|A|+|B|-1),$ where $A,B$ are nonempty subsets of a group
with a prime order, proved by Cauchy \cite{cauchy} and rediscovered
by Davenport \cite{davenport}. Some of the generalizations of this  result
are due to Chowla \cite{chowla}, Shepherdson
\cite{sheph}, Mann \cite{manams} and Kemperman \cite{kempcompl}.

Kneser's generalization of the Cauchy-Davenport Theorem is a basic
tool in Additive Number Theory:

\begin{theorem}[Kneser \cite{knesrdensite} ]\label{kneserth}
Let $G$ be an abelian group and let $A, B\subset G$ be finite
nonempty subsets such that $AB$ is aperiodic. Then $|AB|\ge |A|+|B|-1$.
\end{theorem}

In \cite{knesrdensite} Kneser gives hints for the proof of his theorem.
A  continuous generalization of this result
is proved by Kneser in  \cite{knesrcomp}.
Other proofs of Kneser's Theorem may be found in \cite{manlivre,natlivre,tv}.
Among the numerous applications of Kneser's Theorem, we mention a result of Dixmier
on  the Frobenius problem  \cite{dixmier}.
 Several attempts were made to generalize Kneser's Theorem to non-abelian groups.
 The first result in this direction  is due to Diderrich:

\begin{theorem}[Diderrich \cite{diderrich}]\label{diderrichth}
Let $G$ be a group and let $A, B\subset G$ be finite
subsets such that $AB$ is not the union of left cosets. Assume moreover that the
elements of $B$ commute. Then $|AB|\ge |A|+|B|-1.$
\end{theorem}

It was observed in \cite{hdiderich} that this generalization is
equivalent to Kneser's Theorem, c.f. Corollary \ref{didericheq}.
More investigations  and some examples, showing that the natural extension to the non-abelian case fails to hold, can be found in  Olson \cite{olsonsdif}.

The critical pair Theory is  the description of the subsets  $A,B$ with $|AB|=
|A|+|B|-1$. Vosper's Theorem \cite{vosper1,vosper2} states that in a
group with a prime order $|AB|=|A|+|B|-1\le |G|-2$ holds if and only
if $A$ and $B$ are progressions with the same ratio, where $\min
(|A|,|B|)\ge 2$.
Other proofs of Vosper's Theorem may be found in \cite{manlivre,natlivre,tv}.
More recently the authors of
\cite{hrodseth2} obtained a description of sets $A,B$ with
$|AB|=|A|+|B|\le |G|-4,$ if $|G|$ is a prime.

The last result was applied
to sum-free sets in \cite{klfree}, and  to show the existence rainbow solutions of linear equations in \cite{rainbow}.

 Kemperman's critical pair Theory
\cite{kempacta} provides a generalization of Vosper' Theorem to general abelian groups.

In the non-abelian case only few results were known until recent years.
These results are  due to Kemperman \cite{kempcompl},
Olson \cite{olsonaa,olsonjnt,olsonsdif} and Brailowski-Freiman
\cite{brailowski}.

The results described above were proved using the transformations
 introduced by Cauchy \cite{cauchy}, Davenport \cite{davenport},
 Dyson \cite{dyson} and  Kemperman
\cite{kempcompl}.

The basic properties of  the first three transformations
are given in the books \cite{manlivre,natlivre,tv}.

More recently  K\'arolyi \cite{karolyi} used group extensions and
the Feit-Thompson Theorem to obtain a generalization of Vosper's
Theorem to the non-abelian case.

The exponential sums
 method in Additive Number Theory  gives  some sharp estimates for  $|AB|$ in the abelian
case if $|A|,|B|$ are relatively small. The reader may find
applications of this method in the text books \cite{natlivre,tv} and
the papers
 of  Deshouillers-Freiman \cite{desfreim} and  Green-Ruzsa \cite{greenruz}.

Another method in Additive Number Theory based on Nonstandard Analysis was introduced by Jin.
An example of the application of this method may be found in \cite{jin}.

In this paper we are concerned with the isoperimetric method
introduced by the author in \cite{hcras,hejc2,halgebra,hast}. Let us
present briefly some special cases of this method:

 Let $\Gamma =(V,E)$ be a finite
reflexive relation and consider the objective function $X\mapsto
|\Gamma (X)\setminus X|$, defined on the subsets $X$ with $|X|\ge k$
and $|V\setminus \Gamma (X)|\ge k.$ The minimal value of this
objective function is the {\em $kth$--connectivity} and a {\em
$k$--atom } is a set with minimal cardinality where the objective
function achieves its minimal value. The main result proved in
\cite{halgebra} implies that distinct $k$--atoms of $\Gamma$
intersect in at most $k-1$ elements or  that distinct $k$--atoms of
$\Gamma ^{-1}$ intersect in at most $k-1$ elements. This result,
which generalizes some previous results of the author
\cite{hcras,hejc2,hast}, has several applications in Additive Number
Theory as we shall see in the present paper.

The strong connectivity, usually
defined in Graph theory as the minimum cardinality of a cutset, coincides
with our first connectivity. The $kth$--connectivity was
introduced in \cite{halgebra} in connection with some additive
problems.

Let $B$ be a subset of a finite  abelian group $G$ with
$1\in B$ and $B\neq G$. As showed in \cite{hejc2}, the main
result proved in \cite{hcras} implies that the objective function
$X \mapsto|(XB)\setminus X|$, defined on the nonempty subsets $X$ with $XB\neq G$, attains its minimal value on a subgroup. If
$|G|$ is a prime this value is necessarily $|B|-1.$ The
Cauchy-Davenport Theorem follows obviously from this fact.

In this paper we shall present  basic facts about the isoperimetric
method. We shall improve some of our previous results and obtain
generalizations and short proofs for several known results. We also
give some new applications.

The reader may find some applications of the isoperimetric method in
Serra's survey \cite{serra}. Also,   Balandraud \cite{balart}
developed an isoperimetric approach to Kneser's Theorem.

 The paper's organization is the following:

In Section 2, we present the terminology. In Section 3,  we
introduce the concepts of $kth$--connectivity, $k$--fragment and
$k$--atom and prove some elementary properties of the $k$--fragments. In
Section 4, we give some basic properties of the intersection of
fragments. The main result of this section is Theorem
\ref{inter2frag} which gives conditions implying that the
intersection of two $k$--fragments is a $k$--fragment. This theorem
generalizes results contained in
\cite{hcras,hjct,hejc2,hast,halgebra}. In Section 5, we obtain the
structure of $1$--atoms and give few  applications. Most of the
results of this section were proved in \cite{hejc2,hast}. We prove
them since they are needed in several parts of this paper in order
to make the present work self-contained. In Section 6, we
investigate the inequality $|AB|\ge |A|+|B|/2$ and its critical
pairs.  In Section 7,
 Proposition \ref{m+3} gives the value of
$\kappa _2$ for  a set with a small
cardinality. As an application we generalize the result of K\'arolyi
\cite{karolyi} mentioned above.
In Section 8, we
determine the structure of the $2$--atoms in the abelian case if $\kappa _2 (S)\le |S|$. This
result extends to the infinite case a previous result of the author
\cite{hactaa}. The proof given here is much easier than our first
proof.
 In Section 9, we  give an upper
bound for the size of a $2$--atom.
As an application we generalize to the infinite case a result proved
in the finite case  by Arad and Muzychuk \cite{arad2}.
In Section 10, we present a new basic tool: the strong isoperimetric property.
 This property will be used in a coming
 paper \cite{hkemp} to deduce Kemperman critical pair Theory from this property of the $2$-atoms.

In the Appendix, we give a simple isoperimetric proof of Menger's Theorem in order to make the present work self-contained.

\section{Terminology and preliminaries}

\subsection{Groups}
Let $G$ be a group and let $S$ be a subset of $G$. The subgroup
generated by $S$ will be denoted by $\subgp{S}$. Let  $ A,B$ be
subsets of $ G $. The {\em Minkowski product} is defined as
$$AB=\{xy \ : \ x\in A\  \mbox{and}\ y\in
  B\}.$$

Let  $H$ be a subgroup. Recall that a {\em left $H$--coset} is a set of the $aH$ for some $a\in G$.
The family $\{aH; a\in G\}$ induces a partition of $G$. The trace of this partition on a subset $A$
will be called a {\em left $H$--decomposition}
of $A$.

Therefore a partition $A=\bigcup \limits_{i\in I}
A_i$  is a { left $H$--decomposition} if and only if  $A_i$ is a nonempty intersection of
some left $H$--coset with $A$ for every $i\in I$.
A {\em right} $H$--decomposition is defined similarly.

Let $X$ be a subset of a group $G$. We write

$\Pi^r(X)=\{x\in G : Xx=X\}$
and $\Pi^l(X)=\{x\in G : xX=X\}.$


Notice that \begin{equation}
X=X\Pi^r(X)=\Pi^l(X)X.\label{Pir}\end{equation}


We use the following well known fact:
\begin{lemma}[\cite{manlivre}, Theorem 1]
Let $G$ be a finite group and let
 $A,B$  be  subsets
 such that $|A|+|B|>|G|$.
 Then $AB=G$.

\label{prehistorical}
 \end{lemma}

\subsection{Graphs}
 Let $V$ be a set and let $E \subset V\times V$.  The relation
$\Gamma = (V,E)$ will be called  a  {\em graph}.
 The elements of  $V$
will be called {\em points} or {\em vertices}.
The elements of  $E$
will be called {\em arcs} or {\em edges}.

The diagonal of $V$ is by definition $\Delta
_V=\{(x,x) : x\in V\}$.
 The graph $\Gamma$ is said to be
{\em reflexive} if $\Delta _V \subset E.$
 The {\em reverse   }  of $\Gamma $ is by definition
$\Gamma ^{-1}=(V,E^{-1})$, where $E^{-1}=\{(x,y) : \   (y,x) \in E\}.$

Let $ a\in V$ and let $A\subset V$. The image of $a$ is by
definition  $$\Gamma (a)=\{x :  (a,x)\in E\}.$$ The image of $A$ is
by definition $$\Gamma (A)=\bigcup \limits_{x\in A} \Gamma (x).$$

 The {\it valency}
of $x$ is by definition $d _{\Gamma}(x)=|\Gamma (x)|$. We shall say that
$\Gamma$ is {\em locally finite} if  $d _{\Gamma} (x)$ is finite for
all $x$. We  put $\delta ({\Gamma})=\min \{d_{\Gamma} (x); x\in
V\}.$  The graph $\Gamma$ will be called {\em regular} with valency
$r$ if the elements of $V$ have the same  valency $r$.


Let $\Gamma =(V,E)$ be a  graph. For $X\subset V$, the {\em
boundary} of $X$ is by definition
 $$\partial _{\Gamma}(X)= \Gamma (X)\setminus X.$$

When the context is clear the reference to $\Gamma$ will be omitted. In this case we write
\begin{itemize}
  \item $\partial _{-}(X)=\Gamma ^{-1}(X)\setminus X$,
\item $X^\curlywedge=V\setminus {(X\cup \Gamma (X))}$,
\item $X^\curlyvee=V\setminus{(X\cup \Gamma ^{-1}(X))}.$
\end{itemize}

Most of the time we shall work with reflexive graphs. In this case
we have  $ \Gamma (X)=X\cup \Gamma (X).$

Notice that there is no arc connecting $X$ to $X^{\curlywedge},$
since any arc starting in $X$ must end  in $X\cup
\partial(X).$ The reader should always have in mind this obvious
fact.

Let $\Gamma =(V,E)$ be a reflexive graph. We shall say that a subset
$X$ induces a {\em $k$--separation} if $ k\le \min (|X|,|X^\curlywedge|)<\infty$. We shall say that $\Gamma$ is
$k$--separable if there is a subset $X$ which induces a $k$--separation.

Observe that for every $k,$ $\Gamma$ is $k$--separable if $V$ is
infinite.

Notice that $X$ induces a $k$--separation of $\Gamma$ if and only if
$X^\curlywedge$ induces a $k$--separation of $\Gamma ^{-1}$. In
particular $\Gamma$ is $k$--separable if and only if $\Gamma ^{-1}$
is $k$-separable.

A subset $X$ such that $\Gamma (X)\subset X$ is called a {\em source} of the graph $\Gamma$.
A subset $X$ such that $\Gamma ^{-1} (X)\subset X$ is called a {\em sink} of the graph $\Gamma$.

 A set $T$ of the form $\partial (F)$, where $F\neq \emptyset$
and $F^{\curlywedge}\neq \emptyset$ is called a {\em cutset}. Notice that the
deletion of $T$ destroys all the arcs connecting $F$ to $F^{\curlywedge}$.

Intuitively speaking $\Gamma$ is $k$--separable if there is a
cutset(namely $\partial (X)$) whose deletion creates a source  $X$ of
size $\ge k$ and a sink  ( namely $X^\curlywedge$) of size $\ge k$.

\subsection{Cayley graphs}

 Let $\Gamma=(V,E)$, $\Phi=(W,F)$ be two graphs. A map $f : V  \mapsto W$
will be called a {\em homomorphism }if $(f(x),f(y))\in F$ for all
$x,y\in V$ such that $(x,y)\in E$.

The graph $\Gamma$ will be called {\em point-transitive}  if for all
$x,y\in V$, there is an automorphism $f$ such that $y=f(x)$. Clearly
a point-transitive graph is regular.

Let $G$ be a group
   and let  $ a \in G $. The permutation $ \gamma_a : x \mapsto ax
$ of $G$ will be called   { \em left- translation}. Let  $S$  be a
subset of   $G$. The graph $(G,E),$ where  $ E=\{ (x,y) : x^{-1}y \
\in S \}$ is called a {\it Cayley graph}.  It will  be denoted by
$\mbox{Cay} (G,S)$.

Let $\Gamma =\mbox{Cay} (G,S)$   and  let   $F \subset G $.
Clearly
 $\Gamma (F)=FS $. In the case of a Cayley graph, we shall write $X^S$ instead of
 $X^\curlywedge$. More precisely we put $$X^S=G\setminus (XS).$$

The following facts are easily seen:

\begin{itemize}
  \item  $ (\mbox{Cay} (G,S))^{-1} =\mbox{Cay} (G,S^{-1});$
  \item     For every $a \in G $,  $\gamma_a$ is an
 automorphism of $ \mbox{Cay} (G,S),$ and hence $  \mbox{Cay}
  (G,S ) $  is point-transitive.
\end{itemize}

\section{The isoperimetric method revisited}

In this section, we introduce the concepts of $kth$--connectivity,
$k$--fragment and $k$--atom. We also prove some elementary properties of
these objects.

 Let $\Gamma =(V,E)$ be a locally finite $k$--separable reflexive  graph.
 The {\em $kth$--connectivity}
of $\Gamma$ (called  {\em $kth$--isoperimetric number} in
\cite{halgebra})
 is defined  as

\begin{equation}
\kappa _k (\Gamma )=\min  \{|\partial (X)|\   :  \ \
\infty >|X|\geq k \ {\rm and}\ |V\setminus  \Gamma(X)|\ge k\}.
\label{eqcon}
\end{equation}

 A finite subset $X$ of $V$ such that $|X|\ge k$,
$|V\setminus \Gamma (X)|\ge k$ and $|\partial (X)|=\kappa
_k(\Gamma)$ is called a {\em $k$--fragment} of $\Gamma$. A
$k$--fragment with minimum cardinality is called a {\em $k$--atom}.
The cardinality of a $k$--atom of $\Gamma$  will be denoted by
$\alpha_k(\Gamma)$.

These notions, which generalize some concepts in \cite{hcras,hjct,hejc2,hast}, were introduced in \cite{halgebra}.

For non--$k$--separable graphs, the notions of connectivity, fragments and atoms were not defined.
In order to formulate logically correct statements without assuming  $k$--separability, we
shall now extend the above notions to non $k$--separable graphs by convention:

Let $\Gamma =(V,E)$ be a non $k$--separable graph with $|V|\ge
2k-1$. Then $\Gamma$ is necessarily finite. We put in this case
$\kappa _k(\Gamma)=|V|-2k+1.$ In this case, any set with cardinality  $k$ will be called a {\em $k$--fragment}
and a  {\em $k$--atom}.

A $k$--fragment of $\Gamma ^{-1}$ will be  called a {\em negative}
$k$--fragment. We use the following notations, where the reference
to $\Gamma$ could be implicit:
\begin{itemize}
\item  $\alpha_{-k}(\Gamma ) =\alpha_k(\Gamma ^{-1}),$
  \item  $\kappa_{-k}(\Gamma ) =\kappa_k(\Gamma ^{-1})$.
\end{itemize}

\begin{lemma}
 Let $\Gamma =(V,E)$ be a locally finite  reflexive graph such that
$|V|\ge 2k-1$. Then  $\kappa _k (\Gamma)$ is the maximal integer $j$
such that for every finite subset $X\subset V$  with $|X|\geq k$,

\begin{equation}
|\Gamma (X)|\geq \min \Big(|V|-k+1,|X|+j\Big).
\label{eqisoper0}
\end{equation}
\end{lemma}

Formulae (\ref{eqisoper0}) is an immediate consequence of the definitions. We shall
call (\ref{eqisoper0}) the {\em isoperimetric inequality}. The
reader may use the conclusion of this lemma as a definition of
$\kappa _k (\Gamma)$.

{\bf Remark}. For any locally finite reflexive  graph $\Gamma
=(V,E)$ with $|V|\ge 1$, we have  $\kappa _1(\Gamma)\le \delta ({\Gamma})-1.$

\begin{lemma} {Let $\Gamma =(V,E)$ be a reflexive
  finite graph with $|V|\ge 2k-1$. Then
\begin{equation}\kappa _k =\kappa _{-k}.\label{eqkappa-}
\end{equation}

\label{finiteduality} }
\end{lemma}
\begin{proof}

As observed above, $\Gamma$ is $k$-separable if and only if $\Gamma ^{-1}$
is $k$-separable.
So (\ref{eqkappa-}) holds by convention if
 $\Gamma$ is non $k$--separable.
Suppose now that  $\Gamma$ is $k$--separable, and let $X$ be a
$k$--fragment of $\Gamma$. We have clearly $\partial _{-}
(X^\curlywedge)\subset
\partial (X)$.
Therefore $$\kappa _k(\Gamma )\ge |\partial (X)|\ge |\partial _{-}
(X^\curlywedge)|\ge \kappa _{-k}.$$ The reverse inequality follows
by applying this one to $\Gamma ^{-1}$.
\end{proof}

\begin{lemma} \label{negative}{Let $\Gamma =(V,E)$ be a  locally finite   $k$--separable reflexive
graph. Let   $X$     be a  $k$--fragment. Then

 \begin{eqnarray}
\partial _{-}  (X^\curlywedge)&=&
\partial (X),\label{eqduall}\\
 (X^\curlywedge)^\curlyvee&=&{X}. \label{eqdualf}
\end{eqnarray}
 In particular
$X^\curlywedge$ is a negative $k$--fragment,  if  $V$ is finite.

}

\end{lemma}

\begin{proof}

We have clearly
 $\partial _{-}  (X^\curlywedge)\subset
\partial (X).$

We must have $\partial _{-} (X^\curlywedge)= \partial (X)$, since
otherwise there is $y\in
\partial (X)\setminus
\partial _{-} (X^\curlywedge)$. It follows that
$|\partial (X\cup \{y\})|\le |\partial (X)|-1$, contradicting the
definition of $\kappa _k.$ This proves (\ref{eqduall}).

 We have $\Gamma ^{-1}
(X^\curlywedge)=X^\curlywedge\cup
\partial _{-}(X^\curlywedge) =X^\curlywedge\cup
\partial (X)=V\setminus X.$
This implies obviously (\ref{eqdualf}).

Assume now that $V$ is finite. We have by Lemma \ref{finiteduality},
$|\partial _{-}  (X^\curlywedge)|= |\partial (X)|=\kappa _k=\kappa
_{-k}.$

This proves that $X^\curlywedge$ is a negative $k$--fragment.\end{proof}

We conclude this section by introducing two important notions:

Let $\Gamma =(V,E)$ be a  reflexive graph. We shall say that
$\Gamma$ is a {\em Cauchy graph} if $\Gamma$ if  $\kappa _1= \delta
-1.$
 We shall say that
$\Gamma$ is a {\em Vosper graph} if $\Gamma$ is  non--$2$-separable or $\kappa _2\ge \delta.$

Clearly $\Gamma$ is a Vosper graph if and only if for every
$X\subset V$ with $|X|\ge 2$,
 $$|\Gamma(X)|\ge \min \Big(|V|-1, |X|+\delta\Big).$$

\section{The intersection
of fragments  }

 The main result of this section is
Theorem \ref{inter2frag} which gives conditions implying that the
intersection of two $k$--fragments is a $k$--fragment. Theorem
\ref{inter2frag} implies that  two distinct $k$--atoms intersect in
at most  $k-1$ elements if $\alpha _k\leq \alpha _{-k}$.


\begin{lemma} \cite{hast}\label{partialsub}{Let $\Gamma =(V,E)$ be a  locally finite   reflexive
graph. Let   $X,Y$     be finite nonempty subsets. Then

 \begin{equation}\label{submodularity}
|\partial  (X \cup Y)|+|\partial  (X \cap Y)|\le |\partial  (X )|+|\partial  ( Y)|
\end{equation}}
\end{lemma}

\begin{proof}

 Observe that
\begin{eqnarray*}
|\Gamma (X\cup Y)|&=&|\Gamma (X)\cup \Gamma (Y)|\\
&=&|\Gamma (X)|+|\Gamma (Y)|-|\Gamma(X)\cap \Gamma (Y)|\\
&\le& |\Gamma (X)|+|\Gamma (Y)|-|\Gamma(X\cap Y)|
\end{eqnarray*}
The result follows now by subtracting the equation $|X\cup Y|=|X|+|Y|-|X\cap Y|$.
\end{proof}

The following result is proved in \cite{hast} in the special case
$\kappa _2=\kappa _1$. Indeed the paper \cite{hast} was concerned
only with  Vosper graphs. The concept of $\kappa _k$ was introduced
two years later in \cite{halgebra}.

\begin{theorem}\cite{hast}  Let $\Gamma =(V,E)$ be a reflexive locally finite
$k$--separable graph.   Let $X,Y$ be two fragments of
$\Gamma$ such that $|X\cap Y|\ge k$ and $|X|-|X\cap Y|+k\le |Y^{\curlywedge}|$.

Then  $X\cap Y$ and $X\cup Y$ are $k$--fragments of $\Gamma$.

\label{inter2frag} \end{theorem}

\begin{proof}
Put $\kappa _k=\kappa _k(\Gamma).$
By the definition of $\kappa _k$, we have $|\partial (X\cap Y)|\ge \kappa _k$. Hence we have
by (\ref{submodularity}),
$\kappa _k + |\partial(X\cup Y)|\le 2\kappa _k$. It follows that  $ |\Gamma(X\cup Y)|=|X\cup Y|+ |\partial(X\cup Y)|\le |V|-k$.

By (\ref{submodularity}),
$$2 \kappa _k \le  |\partial(X\cup Y)|+ |\Gamma(X\cup Y)|=|X\cap Y|\le  |\partial(X)|+|\partial ( Y)|=2\kappa _k.$$
It follows that $X\cap Y$ and $X\cup Y$ are $k$--fragments of $\Gamma$.\end{proof}

The next consequence of Theorem \ref{inter2frag} will be a main tool
in this paper.

\begin{theorem} \cite{halgebra}{ Let $\Gamma =(V,E)$ be a reflexive locally finite
$k$--separable graph. Also assume that either $V$ is infinite or
$\alpha _k \le \alpha _{-k},$

Let $A$ be a $k$--atom and let
   $F$   be a   $k$-fragment such that  $|A\cap F|\ge k$.
Then  $A\subset F$

In particular two distinct $k$-atoms intersect in  at most $k-1$
elements. \label{gintersection}} \end{theorem}
\begin{proof}
Let $A'$ be a negative atom. We shall show that   $|F^{\curlywedge}|\ge |A|.$ This holds clearly if
$V$ is infinite. Suppose that $|V|$ is finite. We have now $|F^{\curlywedge}|\ge |A'|\ge |A|,$
by Lemma \ref{negative}. By Theorem
\ref{inter2frag}, $A\cap F$ is a $k$--fragment. By the minimality of $|A|$,
we must have $A\cap F=A$.
\end{proof}

We shall prove a result concerning the intersection of a fragment with
the dual of a negative fragment (a possibly infinite set). In the finite case this result follows
by Theorem \ref{inter2frag}. We used above the submodularity
of $|\partial (X)|$ to prove the intersection property of fragments
as done in \cite{hast}. Here we shall use a intuitive language used
in \cite{halgebra}. The two methods are basically the same.

\begin{theorem} { Let $\Gamma =(V,E)$ be a reflexive locally finite
graph such that $|V|\ge 2k-1.$ Let $X$ be a $k$--fragment and let
   $Y$   be a negative   $k$-fragment such that $|Y|\ge
|X|$ and $|X\cap Y^\curlyvee|\ge k$. Then  $X\cap Y^\curlyvee$ is a
$k$--fragment. In particular  $X\subset Y^\curlyvee$ if $X$ is a
$k$--atom.

\label{gintersectionc}} \end{theorem}

\begin{proof}

The result is obvious if $\Gamma$ is non $k$--separable since a
$k$--fragment is a $k$--subset in this case. So we may assume that
$\Gamma$ is $k$-separable.

{
\begin{center}
\begin{tabular}{|c||c|c|c|c|}
\hline
$\cap $& $Y^\curlyvee$  & $\partial ^-(Y) $ & $Y$ \\
\hline
\hline
$X$&  $R_{11}$ & $R_{12}$ & $R_{13}$ \\
\hline
$\partial (X) $& $R_{21}$& $R_{22}$ & $R_{23}$ \\
\hline
$X^{\curlywedge}$ & $R_{31}$  & $R_{32}$ & $R_{33}$ \\
\hline
\end{tabular}
\end{center}}

By the definition of a $k$--fragment we have
$$ \kappa _k= |\partial (X)|=|R_{21}|+|R_{22}|+
|R_{23}|.$$
The following inclusion follows by an easy verification:
$$ \partial (X\cap Y^\curlyvee) \subset
R_{12}\cup R_{22}\cup R_{21}.$$

We have clearly $|V\setminus \Gamma (X\cap Y^\curlyvee)|\geq
|V\setminus \Gamma (X)|\geq k$. By the definition   we have $|
\partial (X\cap Y^\curlyvee)|\geq \kappa _k.$  It follows that

\begin{eqnarray}
 |R_{21}|+|R_{22}|+|R_{23}| & = &
 \kappa _k \nonumber\\
 & \leq & |\partial (X\cap Y^\curlyvee) |  \nonumber\\
 &\leq & |R_{12}|+| R_{22}|+| R_{21}|.\nonumber
  \end{eqnarray}
  Therefore
\begin{equation}
 |R_{12}|\geq |R_{23}|.\label{r122}
 \end{equation}

Now \begin{eqnarray*} |X^\curlywedge\cap Y|=|R_{33}|
&=& |Y|-|R_{23}|-|R_{13}| \\
&\ge& |X|-|R_{12}|-|R_{13}| \\
&=&|R_{11}|\ge k.
\end{eqnarray*}

 By the definition of $\kappa _{-k}$ we have $$| \partial^{-} (X^\curlywedge\cap Y)|\geq \kappa _{-k}.$$

It follows that

\begin{eqnarray*}
 |R_{12}|+|R_{22}|+|R_{32}| & = &
 \kappa _{-k} \nonumber\\
 & \leq & |\partial ^{}- (X^\curlywedge\cap Y) |  \\
 &\leq & |R_{22}|+| R_{23}|+| R_{32}|.
  \end{eqnarray*}
  Therefore $|R_{12}|\leq |R_{23}|.$  By (\ref{r122}) we have
$|R_{12}|=|R_{23}|.$

It follows that the inequality  $\kappa _k\le
 |\partial (X\cap Y^\curlyvee) |,$ used in the proof of (\ref{r122}), is
an equality and hence  $X\cap Y^\curlyvee$ is a $k$-fragment of
$\Gamma$.\end{proof}

One may define a  {\em cofinite $k$--fragment} as
a subset $X$ with $|X|\ge k$, $|\partial(X)|=\kappa _k$ and
 $\infty >|V\setminus X|\ge k$. This notion allows give  a common proof for Theorems \ref{inter2frag} and \ref{gintersectionc}.
This approach was used in \cite{hast} in a special case and can be adapted very easily to the general case.

\section{ Estimation of the size of a set product }
Most of the results in this section are proved in \cite{hejc2,hast}.
We prove them here since they are needed in several parts of this
paper in order to make the present work self-contained.

Let $G$ be a group and let $S$ be a subset of $G$ with $1\in S.$
We put
\begin{itemize}
\item  $\alpha_k(S)=\alpha_{k}(\mbox{Cay}(\subgp{S},S)) $;
  \item  $\kappa_k(S)=\kappa_{k}(\mbox{Cay}(\subgp{S},S)) $.

\end{itemize}

 We shall say that  a subset $S$
  is {\em $k$--separable} if $\mbox{Cay}(\subgp{S},S)$ is
  {$k$--separable}.
 By a fragment of $S$   we shall mean a fragment of $\mbox{Cay}(\subgp{S},S).$

  We shall say that  a subset $S$
  is a {\em Cauchy subset } (resp. {\em Vosper subset}) if $\mbox{Cay}(\subgp{S},S)$
  is a {Cauchy graph} ( resp. { Vosper
  graph}).
We shall consider only generating subset containing $1$ in order to
avoid  degenerate situations where $\kappa _k=0$.  Notice that
$\kappa_{k}(\mbox{Cay}(G,S) )=0$,  if $S$ generates a finite proper
subgroup. However $\kappa _k(S)>0$ if
$|\subgp{S}|\ge 2k-1.$
 This easy fact was observed in \cite{halgebra}.

Let us prove a lemma:

\begin{lemma} {Let  $ S$  be a  generating subset of a group $G$ with $1\in S$. Let
$H$ be a $k$-atom of $S$ with $1\in H$. Assume that either $G$ is infinite  or $\alpha _k \le \alpha _{-k}.$
 If $|\Pi^r(H)|\ge k$  then $H$ is a subgroup.
\label{rightper}} \end{lemma}
\begin{proof}
Put $Q=\Pi^r(H)$ and  take
$a\in H.$ Since $HQ=H$,  we have using the assumption $1\in H$,
 $aQ \subset aH\cap H.$ By Theorem \ref{gintersection}, $aH=H$.
Then $H^2=H$ and hence $H$ is a subgroup.\end{proof}

The intersection property implies  the following description
of $1$-atoms, obtained in \cite{hejc2} in the finite case.  The general case
was given later in \cite{hast}.

\begin{proposition} \label{Cay}\cite{hejc2,hast}
Let  $ S$  be a finite generating subset of a group $G$ with $1\in S$. Let
$H$ be a $1$--atom of $S$ with $1\in H$.
 Assume that either $G$ is infinite  or $\alpha _1 \le \alpha _{-1}.$
 Then
   $H$ is a subgroup  generated by $S\cap H$.
\end{proposition}

\begin{proof}
Let $a\in H$. The set  $aH$ is a $1$--atom, since any
left-translation is an automorphism of the Cayley graph. Since
$|(aH)\cap H|\ge 1,$ we have by Theorem \ref{gintersection}, $aH=H$.
Then $H$ is a subgroup.

Let $H_0=\subgp{H\cap S}$. We have clearly $H_0S\cap H\subset H_0$.
Therefore  $\partial (H_0)\subset H_0S\setminus H\subset HS\setminus
H.$ It follows that $H_0$ is a $1$--fragment and hence $H_0=H$.
\end{proof}

 \begin{corollary} \label{Cay2}\cite{hejc2,hast}

Let  $ S$  be a  generating subset of a group $G$ with $1\in S$. Let
$H$  be a $1$-atom  such that
$1\in H.$

If  $G$ is infinite  or $\alpha _1 \le \alpha _{-1},$
 then  $H$ is a subgroup.

In particular there is a
finite subgroup $L\neq G$ such that
   $\kappa_1=\min (|LS|-|L|,|SL|-|L|)$.

\end{corollary}

\begin{proof}

Let $K$ be a negative  $1$-atom with $1\in K$.
Assume that either $G$ is infinite  or $|H| \le |K|.$
 By  Proposition \ref{Cay}, $H$ is a subgroup.
 By the
definition of a $1$--atom we have $\kappa _1=|HS|-|H|.$

 Assume now $|H|> |K|$ and that $|G|$ is finite.
By  Proposition \ref{Cay}, $K$ is a subgroup.
By Lemma \ref{finiteduality} and the definition of a negative
$1$--atom, we have $\kappa _1=\kappa _{-1}=|KS^{-1}|-|K|=|SK|-|K|.$\end{proof}

The next lemma could be useful when $S$ generates a proper subgroup:

\begin{lemma}
\label{diderich}

Let $G$ be group and let $A, S$ be finite nonempty subsets of $G$
with $1\in S$. Put $K=\subgp{S}$ and let $A=\bigcup \limits_{i\in I}
A_i$ be a left $K$--decomposition of $A.$ Put $W=\{i :
|A_iS|<|K|\}.$ Then

\begin{equation}\label{nongenerating}
|W|\le \frac{|AS|-|A|}{\kappa _1(S)}.
\end{equation}

\end{lemma}

\begin{proof}
Put $\kappa _1=\kappa _1 (S)$.
 For each $i$ take $a_i\in {A_i}^{-1}$.

By the isoperimetric inequality, we have $|A_iS|=|a_iA_iS|\ge |A_i|+\kappa
_1$ for all $i\in W$. Then
\begin{eqnarray*}
|AS| &=&\sum\limits_{i\in I} |A_iS|\\ &\ge& \sum\limits_{i\in I\setminus W} |A_i|
+\sum\limits_{i\in W} (|A_i|+ \kappa _1) = |A|+|W|\kappa _1.
\end{eqnarray*}\end{proof}

\section{A universal bound for $\kappa _1$}

We give in this section  a characterization of the sets $S$ with  $\kappa
_1(S)= \frac{|S|}{2}$.

\begin{proposition}\cite{hjct,hast}\label{olson}
     Let  $S$  be a finite generating subset of a
group $G$ with $1\in S$. Let $H$ be a $1$--atom and let $K$ be a negative
$1$--atom such that $1\in H\cap K$. Then
\begin{equation}\label{eqolson}
\kappa _1(S)\geq \frac{|S|}{2}.
\end{equation}
Moreover $ \kappa _1(S)=\frac{|S|}{2}$ holds  if and only if one of
the following holds:

\begin{itemize}
  \item  $|H|\le |K|$ or  $G$ is infinite, and there is a
 $u$  such that $ S=H\cup Hu$;
  \item $G$ is finite and $|H|\ge |K|$ and there is a
 $u$  such that $ S=K\cup uK$.
\end{itemize}

\end{proposition}

\begin{proof}

Assume first that $|H|\le |K|$ or that $G$ is infinite. By Corollary
\ref{Cay2}, $H$ is a subgroup. We have

 $\kappa_1(S)=|HS|-|H|\ge\frac{|HS|}{2} \ge
 \frac{|S|}{2}$,
  observing that $|HS|\ge 2|H|$ since $S$ is a generating subset with $1\in S$.
  Suppose now that $ \kappa _1(S)=\frac{|S|}{2}$. We see that $|HS|=2|H|,$ and hence there is a
   $u$  such that $ S=H\cup Hu$.

Assume now that $|H|\ge |K|$ and that  $G$ is finite.  By
Corollary \ref{Cay2}, $K$ is a subgroup. By Lemma
\ref{finiteduality}, we have

 $\kappa_1=\kappa_{-1}=|KS^{-1}|-|K|\ge\frac{|SK|}{2} \ge
 \frac{|S|}{2}$,
  observing that $|SK|\ge 2|K|$ since $S$ is a generating subset with $1\in S$.
  Suppose now that $ \kappa _1(S)=\frac{|S|}{2}$. We see that $|SK|=2,$ and hence there is a
   $u$  such that $ S=K\cup uK$.\end{proof}

   Notice that the
bound $\kappa _1(\Gamma)\ge \frac{\delta(\Gamma)}{2}$  holds for all point-transitive reflexive graphs. This was
proved  in \cite{hjct} in the finite case. The general case is given in
\cite{hast}.

Z\'emor  constructed
 in \cite{zemordam} a Cayley graph with $\alpha _1>\alpha _{-1}$.
The above result suggests more constructions of this type:

{\bf Example} Consider a finite group $G$ of odd order and consider
a non normal subgroup $H$ and an element $u$ such that $uH\neq
Hu.$ Put $S=H\cup Hu$. By Proposition \ref{olson}, $H$ is a $1$-atom
of $S$. Let $Q$ be a negative $1$-atom. $Q$ can not be a subgroup
since otherwise  $|QS^{-1}|=2|Q|=\kappa _1=2|H|$, which is impossible.

\begin{corollary} [Olson]\cite{olsonaa,olsonjnt}
\label{olsonnth} Let  $A,B$  be  finite nonempty subsets of a group
$G$ and put $K=\subgp{BB^{-1}}.$ Then

  \begin{eqnarray}
  |B^j|&\ge& \min(|K|,\frac{(j+1)|B|}{2}), \mbox{and}\label{olsaa}\\
|AB|&\ge& \min(|AK|,|A|+\frac{|B|}{2}.)\label{olsjn}
  \end{eqnarray}

\end{corollary}
\begin{proof}
It is enough to prove Formulae (\ref{olsjn}).

Take $b\in B^{-1}$ and put $S=Bb$. Since $BB^{-1}=SS^{-1}\subset \subgp{S}$  and $S=Bb\subset BB^{-1}$
we have $K=\subgp{S}.$

Take a  left $K$--decomposition $A=\bigcup \limits_{i\in I} A_i.$
Suppose now that $AB\neq AK.$ Then there is an $s$ such that
$|A_sB|<|K|$. Take a $u\in A_s^{-1}$. By the isoperimetric
inequality and by (\ref{eqolson}), $|A_sB|=|uA_sB|\ge
|A_s|+\frac{|B|}{2}.$ Then
$$|AB|=\sum\limits_{i\in I} |A_iB|\ge \sum\limits_{i\in I\setminus
\{s\}} |A_i|+ |A_sB|\ge |A|+\frac{|B|}{2}.$$\end{proof}

Formulae (\ref{olsaa}) is proved by Olson in \cite{olsonaa} as main
tool in his proof that a subset of a finite group $G$ with
cardinality $\ge 3\sqrt{|G|}$ contains some nonempty subset
$\{a_1,\cdots,a_k\}$ with  $a_1\cdots a_k=1$, a result conjectured
by Erd\H{o}s and Heilbronn \cite{erdheil}. Olson's last result improves
results by Erd\H{o}s-Heilbronn \cite{erdheil} and {Szemer\'edi}
\cite{szemeredi}.

Formulae (\ref{olsjn}) is proved by Olson in \cite{olsonjnt}.
Notice that the
bound $\kappa _1(\Gamma)\ge \frac{\delta(\Gamma)}{2}$   for all point-transitive reflexive graphs
proved independently by the author in \cite{hjct} implies easily Olson's result.
Our isoperimetric proof  looks much easier than the proof of Olson \cite{olsonjnt}.

Applications of this formulae to $\sigma$--finite groups are given
by the authors of \cite{hrodseth1} and by Hegyv\'ary
\cite{hegyvary}. This result has the following easy consequence proved
independently by R{\o}dseth \cite{rodsethfrob} and the author \cite{hejc1}:

\begin{corollary} \cite{hejc1,rodsethfrob} \label{orderbase}
Let  $S$  be   generating  subset of a finite group
$G$ with $1\in S$. Put
$|G|=n$   and  $|S|=k.$
 Then
$$S^{\lfloor \frac{2n}{k}\rfloor-1} =G.$$
\end{corollary}

\begin{proof}
Put $j={\lfloor \frac{2n}{k}\rfloor-1}$ and suppose that $S^j\neq
G$. By Lemma \ref{prehistorical}, {$|S^{j-1}|+|S|\le n$}. By
(\ref{olsjn}), we have $|S^{i}|\ge |S^{i-1}|+\frac{k}2$, for all $i\le j-1$. By
iterating we have
$$|S^{j-1}|\ge |S|+(j-2)\frac{k}2=j\frac{k}2.$$ Therefore $ n\ge
|S|+j\frac{k}2$. Hence  $(j+2)k\le 2n$ and $j\le \frac{2n}{k}-2$, a
contradiction.
\end{proof}

Corollary \ref{orderbase} was proved independently by R{\o}dseth
\cite{rodsethfrob} and the author \cite{hejc1}. Actually  a
more general result dealing with graphs having a transitive group
of automorphisms is also proved in \cite{hejc1}.
 An application of Corollary \ref{orderbase} to the Frobenius
problem is given by R{\o}dseth \cite{rodsethfrob}. Also this
corollary is used in \cite{pyber,segal1,segal2} in the investigation
of finitely generated profinite groups.

 Lemma \ref{diderich} has the following consequence:
\begin{corollary}\cite{hdiderich}
\label{didericheq}
Diderrich Theorem \ref{diderrichth}  is equivalent to Kneser's Theorem \ref{kneserth}.

\end{corollary}
\begin{proof}
Diderrich Theorem implies clearly Kneser's Theorem. Suppose now that
any two elements of $B$ commute and let $K$ be the subgroup
generated by $B$. Observe that $K$ is abelian.

 Without loss of generality we may assume that
$1\in B.$ Take a left $K$--decomposition $A=\bigcup \limits_{i\in I}
A_i$.  By Proposition \ref{olson} and Lemma \ref{diderich}, there is
a $j\in I$ such that $A_iB=A_iK$ for all $i\neq j.$ Since $\Pi
^r(A_jB)\subset K$, we must have $|\Pi ^r(A_jB)|=1.$ Take $a\in
A_j$.

Since
$K$ is abelian we have by Kneser's Theorem \ref{kneserth}, $|A_jB|=
|a^{-1}A_jB|\ge |A_j|+|B|-1.$ Therefore $|AB|= \sum \limits_{i\in
I} |A_iB|\ge \sum \limits_{i\in
I\setminus \{j\}} |A_i|+|A_jB|\ge |A|+|B|-1 .$
\end{proof}

\section{Some structural properties of the $2$--atoms}

In this section we obtain some results dealing with non
necessarily abelian groups. Proposition \ref{m+3} gives
the value of $\kappa _2$ for sets with a small cardinality.

We shall define the {\em defect} of $S$ as $$\mu (S)=\kappa
_2(S)-|S|.$$

The
following lemma gives bound on the cardinality of a $2$--atom.
 This bound  allows to give some proofs by induction.

\begin{lemma}\cite{halgebra} \label{small2atom}
Let  $ S$  be a finite generating   subset of a group $G$ with $1\in
S$ and $|S|\ge 3.$ Also assume that either $G$ is infinite  or $\alpha _2 \le \alpha _{-2}.$
  Let $H$ be  a   $2$--atom  with $1\in H$ and $|\Pi ^l(H)|=1$.  Then
  $|H|\le |S|-1$.

\end{lemma}

\begin{proof} Assume  that $|H|\ge |S|.$ Then $|H|\ge 3$.   For each $x\in H$,
there is an $a_x\in H$ such that $a_x^{-1}x\in S\setminus \{1\},$
since otherwise $\partial (H\setminus \{x\})\subset \partial (H),$
and $H\setminus \{x\}$ would be a $2$-fragment, a contradiction.
 By the pigeonhole
principle there are
 $x,y\in H$
with $x\neq y$ and $a_x^{-1}x=a_y^{-1}y$. It follows that $a_y,y \in
a_ya_x^{-1}H$. Therefore by Theorem \ref{gintersection},
$H=a_ya_x^{-1}H$ and hence  $|\Pi^l(H)|\ge 2$, a contradiction.\end{proof}

Put $p(G)=\min \{|H| : H \mbox{ is a  subgroup with}\ |H|\ge 2\}$.

 The reader may define left and right progressions, c.f.\cite{hast}.
 But these notions coincide if the progression contains $1$. We shall
 mean by an $r$--progression a subset of the form $\{r^{j}, \cdots
 ,r^{k+j}\}$.

\begin{proposition} { \label{m+3}
Let  $ S$  be a  generating subset of a group $G$ with $1\in S$ and
$|S|\le p(G)$. Then

\begin{itemize}
  \item[(i)] $S$ is a Cauchy subset,
  \item [(ii)] if $\kappa _2(S)= |S|-1$, then $S$ is a
progression.
\end{itemize}

}

\end{proposition}

\begin{proof}
By  Corollary \ref{Cay2}, there is a finite subgroup $N$ with
$\kappa _1(S)=\min(|SN|,|NS|)-|N|\ge |N|$.  We must have $|N|=1$, since otherwise $\kappa
_1(S)\ge |N|\ge p(G)\ge |S|,$ a contradiction. Therefore $\kappa
_1(S)=\min(|NS|,|SN|)-1= |S|-1,$ and hence (i) holds.

  We shall now prove that $\alpha _2(S)=2$ if $\kappa _2(S)=|S|-1$,
by induction on $|S|.$ Take a $2$-atom $H$ of $S$ and a
$2$-atom $K$ of $S^{-1}$ with $1\in H\cap K.$

{\bf Case 1}. $|H|\le |K|$ or $G$ is infinite.

We have $|\Pi^l(H)|=1,$ since otherwise, we would have
 $|S|-1=\kappa_2(S)=|\Pi^l(H)HS|-|\Pi^l(H)H|\ge
 |\Pi^l(H)|\ge p(G)$, a contradiction.

Put $L=\subgp{H^{-1}}$. We have clearly $L=\subgp{H}$.

By Lemma \ref{small2atom}, $|H|\le |S|-1\le p(G)-2.$ By (i), $H^{-1}$
is a Cauchy subset. Let $S=\bigcup \limits_{i\in I}
S_i$ be a right $L$--decomposition of $A.$ Put $W=\{i :
|HS_i|<|L|\}.$ By Lemma \ref{diderich}, $|W|\le 1$. Then $|I|=1$, since otherwise $|HS|\ge p(G)+|H|>|H|+|S|-1$, a contradiction.
Hence   $L=G$. Now we have
 $\kappa _2(H^{-1})\le  |S^{-1}H^{-1}|- |S^{-1}|\le |H|-1.$

 By the induction hypothesis   there is a   $u$ such that
 $1+|H|=|\{1,u\}H^{-1}|=2|H|-|H\cap Hu^{-1}|$. Since $1\in H$, we have $H\subset
 \subgp{u}$.
   It follows that  $|H\cap {u}H|=|H|-1$. Since $|\Pi^l(H)|=1$, we have by
   Theorem \ref{gintersection},
  $|H|=2.$

Put $H=\{1,r\}.$
    It follows since $|S|\le p(G)\le |\subgp{r}|$ that $S$ is an  $r$--progression.

{\bf Case 2}. $|H|> |K|$ and $G$ is finite. By Lemma
\ref{finiteduality}, we have $\kappa _{2}(S)=\kappa _{-2}(S).$ The
proof follows as in Case 1.\end{proof}

\begin{corollary}\label{karolyi}
Let $G$ be a group and let $A,B$ be subsets of $G$ such that
 $|A|,|B|\ge 2,$ $1\in A\cap B$ and $|B|\le
p(G).$

Assume that $|AB|=|A|+|B|-1\le |K|-1,$ where $K=\subgp{B}$.
\begin{itemize}
  \item If $|A|+|B|= |K|$ then there is an $a$ such that
  $A^{-1}a=K\setminus B;$
  \item If $|A|+|B|\le |K|-1$ then there is an $r$ such that  $B$ and $A$  are $r$--progressions.
\end{itemize}

\end{corollary}
\begin{proof}
 By Proposition \ref{m+3}, $\kappa _1(B)=|B|-1.$
Take a
left $K$-decomposition $A=A_1\cup \cdots \cup A_j$.
 By Lemma \ref{diderich}, $|\{ i : |A_iB|\neq |K|\}|\le 1.$  We have  $|K|-1\ge |AB|>(j-1) |K|,$
  and hence $j=1$.

 Assume first $|A|+|B|= |K|.$
 By the isoperimetric inequality $|A|+|B|-1\ge |AB|\ge |A|+|B|-1=|K|-1.$  Take
 $\{a\}=K\setminus AB.$ We have $A^{-1}a\subset K\setminus B.$
 Since these two sets have the same cardinality we must have $A^{-1}a= K\setminus B.$

Assume now  $|A|+|B|\le |K|-1.$

 Then $B$ is $2$-separable
 and $\kappa _2(B)\le |B|-1<p(G).$ By Proposition \ref{m+3},
  $B$ is a $r$--progression, for some $r$. Now $r$ generates
  $K$, and hence $K$ is a cyclic group. It follows that $|X\{1,r\}|\ge |X|+1$ for every subset $X\subset K$ with $X\{1,r\}\neq K$. Moreover equality holds only if $X$ is an $r$--progression. Now $|AB|=|A\{1,r\}^k|$, where $k=|B|-1$. The
  above observation shows that $|A\{1,r\}|=|A|+1,$ and hence $A$ is an $r$--progression.
 \end{proof}

 Corollary \ref{karolyi} implies  a
result due to Brailowski-Freiman \cite{brailowski} since
$p(G)=\infty$ for a torsion free group. Corollary \ref{karolyi}
implies the validity of K\'arolyi's Theorem 4 \cite{karolyi} for
infinite groups. In the finite case  K\'arolyi's condition
$|A|+|B|\le p(G)$ is relaxed in   Corollary \ref{karolyi} to the
weaker  one $|B|\le p(G)$ and $|A|+|B|\le
|\subgp{B}|-1.$ Notice that Corollary \ref{karolyi} implies Vosper's
Theorem.

\section{ $2$--atoms in abelian groups}

In this section we  determine the structure of the $2$--atoms
in the abelian case if $\mu \le 0$. This result extends to the
infinite case a previous result \cite{hactaa}. The proof given here
is much easier than our first proof.

\begin{lemma}\label{remab}
Let $S$ be a finite subset of an abelian group $G$ with $1\in S.$
Then $\kappa _k(S)= \kappa _{-k}(S)$ and $\alpha _k(S)= \alpha
_{-k}(S)$.
\end{lemma}
\begin{proof}
This follows
 since the map "$x\mapsto x^{-1}$ " is an isomorphism
from $\mbox{Cay} (G,S)$ onto its reverse $\mbox{Cay} (G,S^{-1})$, in
the abelian case.\end{proof}

\begin{lemma} \label{deg2atom}
Let  $ S$  be a finite generating   subset of an abelian  group $G$
with $1\in S$  and $\mu(S)\le 0$.  Let $H$ be  a
$2$--atom of $S$ with $1\in H$. Also assume that $|H|\neq 2$ and that
$H$ is not a subgroup.
 Then\begin{itemize}
       \item $|H|\le \kappa _2 (H)$.
       \item If $H$ generates $G$ then $|H|=3$.
     \end{itemize}

\end{lemma}

\begin{proof}

By Lemma \ref{rightper}, $H$ is aperiodic.
By Lemma \ref{remab} and Theorem \ref{gintersection}, for every
$a\in H\setminus \{1\},$ we have \begin{equation} \label{eqHAP}
|aH\cap H|=1.\end{equation}
Let $K$ be a $2$--atom of $H$ with $1\in K$. Put $K=\{1,a_1, \cdots ,a_{k-1}\}$.
Assume first that $|K|>|H|$.  By Lemma \ref{small2atom},  $K$
is periodic. Hence $K$ is a subgroup by Lemma \ref{rightper} and hence
 $|\kappa_2(H)=|KH|-|K|\ge |K|>|H|$.

Assume now that  $|K|\le |H|$.

 We  have using (\ref{eqHAP})
 \begin{eqnarray}
 \kappa _2(H)+|K| &\ge& |KH|\nonumber \\&\ge& |\{1,a_1, \cdots ,a_{k-1}\}H|\nonumber\\
 &\ge& |H|+|a_1H\setminus H|+ \cdots +a_{k-1}H\setminus (H\cup a_1H\cup \cdots ,a_{k-2}H) \nonumber\\
 &\ge& |H|+|H|-1+ \cdots |H|-k+1=k|H|-\frac{k(k-1)}2 \label{2atomk}
\end{eqnarray}
and hence $\kappa _2 (H)\ge k|H|-\frac{k(k-1)}2-|K|.$ If $k=2$, then $\kappa _2 (H)\ge 2|H|-1-|K|\ge |H|.$
If $k\ge 3$, then $\kappa _2 (H)\ge (k-1)|H|-\frac{k(k-1)}2\ge 2|H|-3\ge |H|.$

Assume now that $H$ generates $G$. The inequality $|HS|\le |H|+|S|\le |G|-2$ implies that
$\kappa _2(H)\le |H|$. Then we have since $|K|\le |H|$ and by (\ref{2atomk}),  $|H|+|K|\ge k|H|-\frac{k(k-1)}2.$

{\bf Case} 1. $k\le |H|-1$. Then $(k-1)(k+1)\le (k-1)|H|\le k+\frac{k(k-1)}2,$ and hence $|H|=k+1=3$.

{\bf Case} 2. $k=|H|$. Then $(k-2)k\le \frac{k(k-1)}2,$ and hence $|H|=k=3$.\end{proof}

The next result describes the $2$--atoms if $\mu (S)\le 0$.
\begin{theorem} { \cite{hejc3,hactaa}. Let  $S$  be
 a finite generating $2$--separable subset of an abelian group $G$ with $1\in S$ and $\mu(S) \leq 0$.
 Also assume that $|S|\neq |G|-6$ if
$\mu (S)=0$.
 Let
 $1\in M$ be a $2$--atom which  is not a subgroup. Then  $|M|=2.$
\label{k=d} }
\end{theorem}

\begin{proof}
We shall write $o(y)=|\subgp{y}|$, for any $y\in G$.
Put $L=\subgp{H}$ and $t=\mu (S)$. Let $H$ denotes a translate of $M$ with the form $\{1,a,b\}$ maximizing
 $o(a^{-1}b)$. Clearly $H$ is not a subgroup.  We  have $\kappa _2 (S) \geq
|S|-1$, by Lemma \ref{remab} and Proposition \ref{Cay}, since in
this case a $2$-atom  is a $1$-atom. Therefore
$$-1\le t \le 0.$$

 Suppose to the contrary that $|H|\ge 3$.

{\bf Case} 1. $L=G$. By Lemma \ref{deg2atom}, $|H|=3$ and $\kappa _2(H)\ge |H|$. Then $|HS|\ge |S|+|H|$, and hence $t=0$.
Consider the case $|S|>|G|-6$. Then $|H^S|\le |G|-|H|-(|G|-5)\le 2.$ It follows by Lemma \ref{negative} that
$H^S$ is a negative $2$--fragment with $|H^S|=2$ and hence $-H^S$ is a $2$--atom, a contradiction.
So we have $|S|\le |G|-7$.

Put $U=\{a^{-1},b^{-1}\}.$
Put $N=(H^2S)\setminus HS$. By the definition of $\kappa _2$, we have
$|HS|+|N|\ge \min(|G|-1,|HS|+3)=|HS|+3$.
Notice that $xH^{-1}$ is a negative $2$--atom.
Take $x\in N$, then $(xH^{-1})\cap (HS)\neq \emptyset$.
Since $(xH^{-1})\not\subset (X^S)$, we have by Theorem \ref{gintersectionc}, $(xH^{-1})\cap (X^S)=\{x\}$.
Therefore $(NU)\subset (HS)\setminus S$.  Since $|(HS)\setminus S|=3$, we have $NU= (HS)\setminus S$ and
$|N|=3$. It follows also that $ab^{-1}N=N$. In particular $o(ab^{-1})=3.$
Since $\{1,a^{-1},ba^{-1}\}$, $\{1,b^{-1}a,ab^{-1}\}$ are translates of $M$, we must have
$\max (o(a),o(b))\le 3.$ It follows that $|G|\le 9$, contradicting the relation $3\le |S|<|G|-6.$

{\bf Case} 2. $L\neq G$. By Lemma \ref{deg2atom}, $\kappa _2(H)\ge |H|$. This implies obviously that $\kappa _1(H)=|H|-1$.

Put $L=\subgp{H}$ and take an  $L$--decomposition $S=S_1\cup \cdots
\cup S_j$. Without loss of generality we may assume $1\in S_1$ and $|HS_1|\le
\cdots \le |HS_j|$. Put $W=\{i : |HS_i|<|L|\}.$

By Lemma \ref{diderich}, $|W|\le \frac{|HS|-|S|}{\kappa _1(H)}\le \frac{|H|}{|H|-1}<2.$
Then we must have $|HS_i|=|L|,$ for all $i\ge 2.$ We must have

$$
LS=G ,$$ since otherwise
 \begin{eqnarray*}
 |S|-1\le \kappa_2(S)&\le &|LS|-|L|\\&=&(j-1)|L|=|HS|-|HS_1|\le |S|-3, \end{eqnarray*} a contradiction.
In particular  \begin{equation}\label{eqLS}S=(G\setminus L)\cup S_1.\end{equation}
Using (\ref{eqLS}) we see that any $2$-- fragment of $S$ containing $1$ is a $2$-fragment
of $S_1$. In particular $H$ is a $2$-atom of $S_1$.

We have  $|S_1|\neq  |L|-6,$ otherwise we have by (\ref{eqLS}), $|S|=  |G|-6.$

We have also $|S_1|\ge  2,$ since otherwise \begin{eqnarray*}
 |HS|&=&(j-1)|L|+|H|\ge |S|-|S_1|+3= |S|+2, \end{eqnarray*} a contradiction.
Now we apply Case 1 to get a contradiction. \end{proof}

The above result was proved for $\kappa _2=|S|-1$ in \cite{hejc3}, and
for finite groups in \cite{hactaa}.
In the case where $|G|$ is a prime, a proof of Theorem \ref{k=d}
using the Davenport's transform was obtained by the authors of
\cite{hrodseth2}. In \cite{SZ}, Serra and Z\'emor proved that a
$2$--atom of $S$ has size $=2$, if $|G|$ is a prime  and $|S|<
{{4+\kappa _2-|S|} \choose {2}}.$ A short proof of the last result
was given by the authors of \cite{hgoaa1}. A generalization of this
result to arbitrary finite abelian groups was obtained by the
authors of \cite{hgochowla}  when $\mu(S)\le 4.$

An example given by Serra and Z\'emor in  \cite{SZ} shows that in the
prime case the $2$-atom may have size $=3$ if $\mu =0$ and $|S|=|G|-6.$

\section{An upper bound for the size of a $2$--atom}

In this section we prove more results on the intersection of fragments that we shall need in the next section.

\begin{lemma}\label{dualfrageq} Let $\Gamma =(V,E)$ be a  locally finite   $k$--separable reflexive
graph. Let   $X$    and  $Y$ be   $k$--fragments. Then $X\subset Y$
if and only if $Y^\curlywedge\subset X^\curlywedge$.
\end{lemma}

\begin{proof}
Assume first that $X\subset Y.$ Then
\begin{eqnarray*}
X^{\curlyvee}=V\setminus  \Gamma  (X)\supset
V\setminus \Gamma  (Y)=Y^{\curlyvee}
\end{eqnarray*}
Assume now that  $Y^{\curlyvee}\subset X^{\curlyvee}.$ Then
\begin{eqnarray*}
X=V\setminus  \Gamma ^{-1} (X^{\curlyvee})\subset
V\setminus \Gamma ^{-1}  (Y^{\curlyvee})=Y.
\end{eqnarray*}
\end{proof}

\begin{lemma}  Let $\Gamma =(V,E)$ be a reflexive locally finite
$k$--separable graph with $k\ge 2$. Let $A,F$ be $k$--fragments such
that
 $|A|\le | F^{\curlywedge}|$ and $|A\cap F|\ge k-1.$
Then
\begin{eqnarray}
|A\cap \partial( F)|&\le& |\partial (A)\cap F^\curlywedge |,  \label{1977}\\
  |\Gamma(A)\cap \Gamma( F)|&\le& |A\cap F |+\kappa _k \ \mbox{and}\label{eq1+kappa}\\
|F^\curlywedge\setminus A^\curlywedge| &\le &|A\setminus F|+\kappa _k-
\kappa _{k-1}.\label{eqAminusF}
\end{eqnarray}
\label{1+kappa} \end{lemma}

\begin{proof}

{
\begin{center}
\begin{tabular}{|c||c|c|c|c|}
\hline
$\cap$& $F$  & $\partial (F) $ & $F^{\curlywedge}$ \\
\hline
\hline
$A$&  $R_{11}$ & $R_{12}$ & $R_{13}$ \\
\hline
$\partial (A) $& $R_{21}$& $R_{22}$ & $R_{23}$ \\
\hline
$A^{\curlywedge}$ & $R_{31}$  & $R_{32}$ & $R_{33}$ \\
\hline
\end{tabular}
\end{center}}

Suppose that (\ref{1977}) is false. Then
$$|R_{12}|> |R_{23}|.$$

It follows  that

 \begin{eqnarray*}
|F^\curlywedge\cap A^\curlywedge|
&=& |F^\curlywedge|-|R_{23}|-|R_{13}| \nonumber\\
&>& |A|-|R_{12}|-|R_{13}| \nonumber\\
&=&|R_{11}|\ge k-1, \label{r333}
\end{eqnarray*}

Now we have
 \begin{eqnarray*}
|R_{32}|+| R_{22}|+| R_{23}|&\ge&|\partial (A\cup F)|\ge \kappa
_k\\&=&|R_{12}|+| R_{22}|+| R_{32}|,
 \end{eqnarray*}

 and hence $|R_{23}|\ge |R_{21}|$ a contradiction proving (\ref{1977}).
 Now we have  \begin{eqnarray*}
 |\Gamma(A)\cap \Gamma( F)|&=& |A\cap F |+
 |R_{12}|+ |R_{21}| +|R_{22}|\\ &\le&
 |A\cap F |+
 |R_{23}|+ |R_{21}| +|R_{22}\\ &=&|A\cap F |+\kappa _k.
\end{eqnarray*}
This proves (\ref{eq1+kappa}).

Since $|A\cap F|\ge k-1$, we have
\begin{eqnarray*} |R_{12}|+| R_{22}|+|
R_{21}|&\ge& \kappa _{k-1}\\&= &
\kappa _k-
(\kappa _k- \kappa _{k-1})\\&=&
|R_{21}|+| R_{22}|+|
R_{23}|-
(\kappa _k- \kappa _{k-1}).
\end{eqnarray*}

It follows that
$$|R_{23}|\le |R_{12}|+\kappa _k- \kappa _{k-1}.$$

Hence
\begin{eqnarray*}
|F^\curlywedge\setminus A^\curlywedge|
&=& |R_{13}|+|R_{23}| \nonumber\\
&\le& |R_{13}|+|R_{12}|+\kappa _k- \kappa _{k-1}\nonumber\\
&=&|A\setminus F|+\kappa _k- \kappa _{k-1}.
\end{eqnarray*}

This proves (\ref{eqAminusF}).\end{proof}


We  shall now investigate the number of $2$--atoms containing
a given element and obtain an upper bound for the size of a
$2$--atom.

We obtain some applications including a generalization to the
infinite case a result proved in the finite case  by Arad and
Muzychuk \cite{arad2}.

The smallest number $j$ (possibly null) such that every element
$x\in V$ belongs to $j$ pairwise distinct $k$--atoms will be denoted
by $\omega _k(\Gamma)$.
 We shall write  $\omega _k(\Gamma^{-1})=\omega _{-k}(\Gamma)$.
 We also write $\omega _k(S)=\omega _{k}(\mbox{Cay}(\subgp{S},S))$.
The next result is a basic tool in the investigation of the
$2$-atoms  structure. The case $\kappa
_2=\kappa _ 1$ of this result is proved in \cite{hast}.
Also the finite case  of this result is proved in \cite{hactaa}.

\begin{theorem} Let $\Gamma =(V,E)$ be a reflexive locally finite graph.  Let $H$ be a $2$--atom and let $K$ be a negative $2$--atom with
 $|K|\ge |H|\ge 3.$
Also assume that $V$ is infinite or
$\alpha _2 \le \alpha _{-2},$. Then one of
the following holds:
\begin{itemize}
  \item [(i)]$\min (\omega_2 ,\omega _{-2}) \le 2,$
  \item [(ii)] $|H|\le 3+\max (\kappa _2-\delta,\kappa _{-2}-\delta _{-}).$

\end{itemize}

\label{superatoms} \end{theorem}

\begin{proof}

Suppose contrary to (i) that $\omega _2\ge 3$ and $\omega _{-2} \ge
3$. We have $\alpha _1 =\alpha _{-1}=1,$ since otherwise a $2$--atom
containing $x$ is a $1$--atom containing $x$ and it is unique by
Theorem \ref{gintersection}, contradicting $\min (\omega_2,\omega _{-2})\ge
3$.

Take $v\in V$ and choose two distinct $2$--atoms $M_1, M_2$ such
that $v\in M_1\cap M_2.$ By Theorem \ref{gintersection},  we have
 $M_1\cap M_2=\{v\}$.

We have $M_1^{\curlywedge}\not\subset M_2^{\curlywedge}$, by Lemma
\ref{dualfrageq}.

Take $w\in M_1{^\curlywedge}\setminus M_2^{\curlywedge},$ and take
three pairwise distinct negative $2$-atoms $L_1,L_2,L_3$ such that
$w\in L_1\cap L_2\cap L_3.$

Assume first that for some $i\neq j$ we   have
 $ L_i\cup L_j \subset M_1^{\curlywedge}.$

 By Theorem \ref{gintersectionc}, $|L_i\cap M_2^{\curlywedge}|\le 1,$
for every $i$.

Then we have using (\ref{eqAminusF}), the fact that $\kappa _1=\delta -1$ and the intersection property of atoms

\begin{eqnarray*} |H|+\kappa _2-\delta &=&|H|+
 \kappa _2-\kappa _1-1\\
  &=&|M_1\setminus M_2|+
 \kappa _2-\kappa _1\\
&\ge&
|(M_1{^\curlywedge}\setminus
M_2^{\curlywedge})\cap (L_i\cup L_j)|\\
&=& |L_i\cup L_j|- |(L_i\cup L_j)\cap M_2^{\curlywedge} |\\
&\ge& |L_i\cup L_j|-2\\
&=&
 2 |K|-3\ge 2|H|-3,
\end{eqnarray*}
and hence  (ii) holds.

We can now assume without loss of generality that

 $$ L_1, L_2 \not\subset M_1^{\curlywedge}.$$

 By Lemma \ref{dualfrageq},  we have  $M_1\not\subset L_i^{\curlyvee}$ for $1\le
i\le 2$.  By Theorem \ref{gintersectionc}, $|M_1\cap
(L_1^{\curlyvee}\cup L_2^{\curlyvee})|\le 2$. Then  $|M_1\cap \Gamma
^{-1}(L_1)\cap \Gamma ^{-1}(L_2)|\ge |M_1|-2.$

Now $ \Gamma ^{-1}(L_1)\cap \Gamma ^{-1}(L_2)\supset (\Gamma
^{-1}(L_1)\cap \Gamma ^{-1}(L_2)\cap M_1) \cup \Gamma ^{-1}(w).$
Notice that $\Gamma ^{-1}(w)\cap M_1=\emptyset.$

Then  we have by  (\ref{eq1+kappa})

\begin{eqnarray*}
\delta _- +|H|-2&\le&  |\Gamma ^{-1}(w)|+|M_1\cap \Gamma
^{-1}(L_1)\cap \Gamma ^{-1}(L_2)|\\ &\le& | \Gamma ^{-1}(L_1)\cap
\Gamma ^{-1}(L_2)| \\ &\le & 1+\kappa _{-2}.
\end{eqnarray*}
Therefore
 $|H| \le  3+\kappa _{-2}-\delta _-.$\end{proof}

Let us apply this result in the symmetric case.

\begin{corollary} 
Let  $S$  be a finite generating  subset of a group $G$ with $1\in
S$ and $S=S^{-1}$. Let $H$  be a $2$-atom of $S$ such that $1\in H$.
If $|H|\ge\kappa_2-|S|+4$, then
 $|\Pi^l(H)|\ge 2$.
\label{Caysym}
\end{corollary}

\begin{proof}
Since $S=S^{-1},$ $H$ is also a negative $2$--atom.
Also
$\kappa _2=\kappa _{-2}$. Take a $3$--subset $\{a_1,a_2,a_3\}$ contained in $H$.

 By   Theorem \ref{superatoms}, $\omega _2\le 2.$ Hence two of the $2$--atoms $a_1^{-1}H,
 a_2^{-1}H, a_3^{-1}H$ must be equal and hence $|\Pi^l(H)|\ge 2$.
\end{proof}

Let $G$ be group and let $S$ be a finite subset with $1\in S$.
One may have $\alpha _1 > \alpha _{-1}$ \cite{zemordam}. We may even have $\kappa _1 >\kappa _{-1}$ if $G$ is infinite.
We have seen that $\kappa _k =\kappa _{-k}$ and that $\alpha _k =\alpha _{-k}$ if $G$ is abelian since the Cayley graph defined by $S$
is isomorphic to its reverse. We shall define  subsets having  this property:

 The set $S$ is
said to be {\it normal} if  $xSx^{-1}=S,$  for every $x\in G$.

 It would be too restrictive to deal only with normal subsets,
since the isopermetric results are valid for translate copies of some set.
We consider the following more general notion:

We shall say
that $S$ is {\it semi-normal} if there is $a\in G$ such that for
$xSx^{-1}=Sa^{-1}xax^{-1},$  for every $x\in G$.

In this case we have $$x(Sa^{-1})x^{-1}=Sa^{-1}xax^{-1}xa^{-1}x^{-1}=Sa^{-1}.$$
In particular $Sa^{-1}$ is normal.
On the other side it is an easy exercise to see that $Xa$ is {\it semi-normal} if $X$ is normal.

\begin{lemma} \label{lem:conjugate}
Let  $ S$  be a finite generating  semi-normal subset of a group $G$
with $1\in S$.  Then the map $Y\mapsto a^{-1}Y^{-1}a$ is a bijection from the set of $k$--fragments of
$S$ onto the set of $k$--fragments of
$S^{-1}$. In particular

 $\kappa _k =\kappa _{-k},$
$\alpha _k =\alpha _{-k}$
and  $\omega _k =\omega {-k}.$

\end{lemma}

\begin{proof} By the definition of a semi-normal subset we have $XS=Sa^{-1}Xa,$ for every subset $a$.
Let $X$ be a subset with $\min (|X|,|G|-|XS^{-1}|)\ge k$.
Then $|XS^{-1}|-|X|=|SX^{-1}|-|X|=|aX^{-1}a^{-1}S|-|X|\ge \kappa _k(S).$
It follows that $\kappa _k\le \kappa _{-k}$.

Similarly for every subset $Y$  with $\min (|Y|,|G|-|YS|)\ge k$.
Then $|YS|-|Y|=|Sa^{-1}Ya|-|Y|=|Y^{-1}aS^{-1}|-|Y|\ge \kappa _{-k}(S).$
It follows also that the map $Y\mapsto a^{-1}Y^{-1}a$ is a bijection from the set of $k$--fragments of
$S$ onto the set of $k$--fragments of
$S^{-1}$. Therefore $\alpha _k =\alpha _{-k}$
and  $\omega _k =\omega _{-k}.$\end{proof}

The next result extends to the infinite case  Arad-Muzychuk Theorem 3.1 of \cite{arad2}. Our terminologies differ slightly:

\begin{corollary} 
Let  $ S$  be a finite generating  semi-normal subset of a group $G$
with $1\in S$. Let $H$  be a $2$-atom of $S$ such that $1\in H$. If
$|H|\ge\kappa_2-|S|+4$, then
   $H$ is  subgroup of $G$ and $[G:N_G(H)]\le 2.$
\label{Caynormalarad}
\end{corollary}

\begin{proof}
By Lemma \ref{lem:conjugate}, $\alpha _2=\alpha _{-2}$, $\kappa
_2=\kappa _{-2}$ and $\omega _2 =\omega _{-2}.$
By   Theorem \ref{superatoms}, $\omega _2\le 2.$

Take a $3$--subset $\{a_1,a_2,a_3\}$ contained in $H$.
  Then two of the $2$--atoms $a_1^{-1}H, a_2^{-1}H, a_3^{-1}H$ must be equal and hence
  $2\le |\Pi^l(H)|=|\Pi^r(H^{-1})|$.
 By Lemma \ref{lem:conjugate}, $H^{-1}$ is a negative $2$--atom.
 By Lemma \ref{rightper}, $H^{-1}$ is a subgroup. Therefore $H$ is a subgroup.

$xS=Sa^{-1}xa,$
 Clearly $|xHx^{-1}S|=|xHSa^{-1}x^{-1}a|=|HS|$.
Therefore $xHx^{-1}$  is a $2$--atom
 for every $x$. If $xHx^{-1}=H,$ for every $x$ then $N_G(H)=G$.
 Suppose that there is $a$ such that $aHa^{-1}\neq H,$

 We have $[G:N_G(H)]\le 2,$ since for every $x\in G$, we have
 $xHx^{-1}=H$ or $xHx^{-1}=aH,$ otherwise $\omega _2\ge 3$, a contradiction.\end{proof}

\section{The strong isoperimetric property}

In this section, we prove the strong ioperimetric  property which  allows to use the structure of atoms to calculate all the
 fragments in some important cases.
 The strong isoperimetric methodology
 will be used in a coming papers
\cite{hkemp,hkempp} to extend Kemperman's critical pair Theory using Theorem \ref{k=d}.

 We shall use a mi-max result proved by   Menger \cite{menger} for symmetric graphs and generalized to arbitrary graphs by Dirac \cite{dirac}.
The Dirac-Menger Theorem is now a basic tool in Additive number Theory \cite{natlivre,tv}. In particular it was used by Ruzsa \cite{ruz} to
give a simple proof of the Plu\"{u}nnecke inequalities. We  need it to prove the strong isoperimetric property. In the Appendix, we give a simple isoperimetric proof of this result.

Let $\Gamma =(V,E)$ be graph and let $a,b\in V$. A {\em path} from $a$ to $b$ is a sequence of arcs $\sigma =\{(x_1,y_1), \cdots ,(x_k,y_k)\}$
with $x_i=y_{i+1}$ for all $1\le i \le k-1$, $x_1=a$ and $y_k=b$. We put $V(\sigma)=\{x_1, \cdots ,x_k,y_k\}$. Two paths $\sigma, \tau $
from $a$ to $b$ will be said to be {\em openly} disjoint if $V(\sigma)\cap V(\tau)\subset \{a,b\}$.

Let $x,y$ be elements of $V$. We shall say that $x$ is {\em
$k$--connected} to $y$ in $\Gamma$ if $ |\partial (A)|\ge k$,
for every subset $A$ with $x\in A$ and $y\notin A\cup \Gamma (A)$.

\begin{theorem} ( Dirac-Menger) \label{menger}

Let $\Gamma=(V,E)$ be a finite reflexive graph Let $k$ be a
nonnegative integer. Let $x,y\in V$ such that $x$ is
$k$--connected  $y$, and $(x,y)\notin E$. Then there are
$k$ openly disjoint paths from $x$ to $y$.
\end{theorem}

One may formulate Menger's Theorem for non reflexive graphs. Such a
formulation follows easily from the
reflexive case.

We need the following consequence of Menger's Theorem:

\begin{corollary} { Let $\Gamma $ be a locally finite   reflexive
graph and let $k$ be a nonnegative integer with
 $k\le \kappa _1$. Let   $X$  a
finite subset of $V$ such that $\min (|V|-|X|, |X|)\ge k.$ Then there are
pairwise distinct elements
 $x_1, x_2, \cdots, x_{k} \in X$ and pairwise distinct elements
 $y_1, y_2, \cdots, y_{k} \notin X$ such that

  $$(x_1,y_1), \cdots , (x_{k}, y_{k})\in E.$$

\label{gstrongip}}
\end{corollary}

\begin{proof}
By the definition of $\kappa _1$, we have $|\partial (Y)|\ge \min
(|V|-|Y|,\kappa _1)\ge k ,$ for every $Y\subset V$.
 Let $\Phi=(\Gamma(X),E')$ be
the restriction of $\Gamma$ to $\Gamma (X)$ (observe that $X\subset
\Gamma (X)$). Choose two elements $a,b\notin V.$ Let $\Psi$ be the
reflexive graph obtained by connecting $a$ to $X\cup \{a\}$ and
$\partial (X)\cup \{b\}$ to $b$.  We shall show that $a$ is {
$k$--connected} to $b$ in $\Psi$. Take $a\in T$ such that
$b\notin \Psi (T)$. Then clearly $T\subset X\cup \{a\}. $ Assume
first  $T=\{a\}$. Then $|\Psi (T)|-|T|=|X\cup \{a\}|-1\ge k.$ Assume
now  $T\cap X\neq \emptyset$. We have  $\Psi (T)= X\cup \{a\}\cup
\Gamma (T\cap X)$.  Therefore

\begin{eqnarray*}
 |\Psi (T)|&\ge& 1+|X|
+|\Gamma  (T\cap X) \setminus X|\ge
1+|X|+
(|T\cap X|+\kappa _1(\Gamma)- |X|)>k.
\end{eqnarray*}

By Menger's Theorem there  are $P_1, \cdots , P_{k }$ openly
disjoint paths from $a$ to $b$. Choose $x_i$ as the last point of
the path $P_i$ belonging to $X$ and let $y_i$ the successor of $x_i$
on the path $P_i$. This choice satisfies the requirements of the
proposition.
\end{proof}

We call the property given in Corollary  \ref{gstrongip}  the {\em
strong isoperimetric property}.

We shall use this property in the special case:

\begin{proposition} { Let $G $ be an abelian group and let $S$ be a finite subset of $G$ with $1\in S$.
Let $H$ be a subgroup of $G$ which is a $2$--fragment of $S$. Let   $S=S_0\cup \cdots \cup S_u$ and $X=X_1\cup \cdots \cup X_t$ be  $H$-decompositions. Also assume that
 $|G|-(t+1)|H|\ge u|H|.$ Then there are
pairwise distinct elements
 $n_1, n_2, \cdots, n_{r} \in [0,t]$ and  elements
 $y_1, y_2, \cdots, y_{r} \in S\setminus H$ such that

 $$|\phi(X\cup (X_{{n_1}}y_1)\cup \cdots  \cup (X_{{n_r}}y_{r})|=t+1+u.$$

\label{astrongip}}
\end{proposition}

\begin{proof}

Let $\phi$ denotes the canonical morphism from $G$ onto $G/H$. Let us show that   $\kappa _1(\phi (S))\ge u$.

Let $Y\subset G/H$ be such that $Y+\phi (S)\neq G$. By the definition we \begin{eqnarray*}
|\phi^{-1}(Y)+S|&\ge &|\phi^{-1}(Y)|+\kappa _2(S)\\&=& |\phi^{-1}(Y)|+u|H|.
\end{eqnarray*}
Therefore $|Y+\phi(S)|\ge |Y|+u$.

The result follows
now by applying Corollary \ref{gstrongip} to $\phi(X)$ and $\phi(S)$.
\end{proof}

Proposition \ref{astrongip} follows easily by Hall marriage Lemma if $X=S$.

\section{Appendix: An isoperimetric proof of  Menger's Theorem}

We present here an isoperimetric proof of Menger's Theorem. Let $E
\subset V\times V$ and let $\Gamma=(V,E)$ be a reflexive graph.  Let $x,y$ be elements of $V$.
 The graph $\Gamma$ will be called {\em
$(x,y)$--$k$--critical} if $x$ is { $k$--connected} to
$y$ in $\Gamma$, and if this property is destroyed by the deletion
of every arc $(u,v)$ with $u\neq v$.

A subset $A$ with $x\in A$ and $y\notin \Gamma (A)$ and $|\partial
(A)|=k$ will be called a {\em $k$--part} with respect to
$(x,y,\Gamma)$.

 The reference to $(x,y)$ will be omitted.

\begin{lemma}\label{frxxy}
Assume that $\Gamma=(V,E)$ is {$k$--critical} and let  $(u,v)\in E$
be an arc with $u\neq v$. Then $\Gamma$ has $k$--part $F$ with $u\in
F$ and $v\in
\partial (F)$.
\end{lemma}

\begin{proof}

Consider the graph $\Psi=(V,E \setminus\{(u,v)\})$.
 There is  an $ F$ with $x\in F$ and $y\notin \Psi (F)$ such that $|\partial _{\Psi}(F)|<k$.
 This forces that  $u\in F$ and that $v \in \partial (F)$, since otherwise $\partial _{\Psi} (F)=\partial
_{\Gamma}(F)$.

Since $\partial _{\Psi} (F)\cup \{v\} \supset \partial
_{\Gamma}(F)$,
 we have  $|\partial _{\Gamma}(F)|\le k$. We must have
 $|\partial _{\Gamma }(F)|= k$, since $u$ is  $k$--connected to $v$
in $\Gamma$. This shows that $F$ is a $k$-part.\end{proof}

\begin{lemma} \label{dualxy}
Let $F$ be a $k$--part with respect to $(x,y,\Gamma)$. Then
$F^\curlywedge$ is a $k$--part a with respect to
$(y,x,\Gamma^{-1})$. Moreover $\partial _{-}  (F^\curlywedge)=
\partial (F).$

 In particular   $y$ is
 $k$--connected to $x$ in  $\Gamma^{-1}$,
  if $x$ is $k$--connected to $y$  in
$\Gamma$.
\end{lemma}
\begin{proof}

We have clearly
 $\partial _{-}  (F^\curlywedge)\subset
\partial (F).$ Put
 $C=\partial (F)\setminus \partial _{-}  (F^\curlywedge).$

 Since $y\notin \Gamma (F\cup C)$,  we have $k\le |\partial (F\cup
C)|\le |\partial _{-} (F^\curlywedge)|\le |\partial (F)|=k$.

\end{proof}

The above lemma is a local version of the isoperimetric duality given in {Lemma} \ref{negative}.

\begin{lemma} \label{frsxy}
Assume that $\Gamma=(V,E)$ is {$k$--critical} and that $\Gamma
(x)\cap \Gamma ^{-1}(y)=\emptyset$. There is a $k$--part $F$ of
$\Gamma$ such that $\min (|F|,|F^\curlywedge|)\ge 2.$
\end{lemma}
\begin{proof}

Take a path  $[x, a,b, \cdots,c,y]$ of minimal length from $x$ to
$y$. By Lemma \ref{frxxy}, there is a $k$-part $F$, with $a\in F$
and $b\in
\partial (F)$. We have $\{x,a\}\subset F$. We have
$|F^\curlywedge|\ge 2$ since otherwise $F^\curlywedge=\{y\}$. Hence
by Lemma \ref{dualxy}, $b \in \partial (F)=\partial ^-(\{y\})$.
Therefore $b\in \Gamma (x)\cap \Gamma ^{-1}(y),$ a contradiction.

\end{proof}

Let $x$ be an element of $V$ and let $T=\{y_1,\cdots ,y_k\}$ be a
subset of $V\setminus \{v\}$. A family of $k$--openly disjoint paths
$P_1, \cdots , P_k$, where $P_i$ is a path  from $x$ to $y_i$ will
be called an {\em $(x,T)$--fan}.

\begin{proofof}{Theorem}{\ref{menger}}

 The proof is by induction, the result being obvious for
$|V|$ small. Assume first that there $z\in \Gamma (x)\cap \Gamma
^{-1}(y)$. Consider the restriction $\Psi$ of $\Gamma$ to
$V\setminus \{z\}$.  Clearly $x$ is {$(k-1)$--connected} to  $y$
in $\Psi$. By the induction hypothesis  there are $(k-1)$--openly
disjoint paths from $x$ to $y$ in $\Psi$. We adjoin the path
$[x,z,y]$ to these paths and we are done. So we may assume that
$\Gamma (x)\cap \Gamma ^{-1}(y)=\emptyset$.

By Lemma \ref{frsxy} there is a part $F$ with  $\min
(|F|,|F^\curlywedge|)\ge 2.$ Consider the reflexive graph
$\Theta=(V',E')$ obtained by contracting $F^\curlywedge$ to a single
vertex $y_0$. We have $V'=(V\setminus F^\curlywedge)\cup \{y_0\}$.
Since $|V'|<|V|$, by the induction hypothesis there are $k$ openly
disjoint paths form $x$ yo $y_0$. By deleting $y_0$ we obtain an
$(x, \partial (F))$--fan. Similarly by contracting $F$ and applying
induction, we form a $( \partial (F),y)$--fan.

By composing these two fans, we form  $k$ openly disjoint paths from
$x$ to $y$.
\end{proofof}


{\bf Acknowledgement}

This paper was revised while the author is visiting Centre de Recerca Matem\`atica, Barcelona, Catalonia, Spain.
The author would like to thank this Institute and Professor Oriol Serra for their hospitality.
Many thanks also for an anonymous referee for many valuable comments.


\begin{thebibliography}{99}


\bibitem{arad2} Z.Arad,  M.Muzychuk, Order evaluation of products of subsets in
finite groups and its applications. II.  {\it Trans. Amer. Math.
Soc.} 349  (1997),  no. 11, 4401--4414.

\bibitem{balart} E. Balandraud, Un nouveau point de vue isop\'erimetrique
appliqu\'e au th\'eor\`eme de Kneser, {\it Preprint}, december 2005.

\bibitem{brailowski} L.V. Brailowski, G. A. Freiman, On a product of
finite subsets in a torsion free group, {\it J. Algebra} 130 (1990),
462--476.



\bibitem{cauchy} A.  Cauchy,  Recherches  sur  les  nombres,  {\it J.  Ecole  polytechnique}
9(1813), 99-116.

\bibitem{chowla} S.~Chowla, A theorem on the addition of residue
classes: applications to the number $\Gamma (k)$ in Waring's
problem, {\it Proc.Indian Acad. Sc.} {\bf 2} (1935) 242--243.




\bibitem{davenport} H. Davenport, On the addition of residue  classes, {\it J.  London  Math.
Soc.} 10(1935), 30--32.


\bibitem{desfreim} J-M. Deshouillers, G. A. Freiman, A step beyond
  Kneser's Theorem. {\it Proc. London Math. Soc.} (3)86 (2003), no 1,
  1--28.


\bibitem{diderrich} G.T. Diderrich, On Kneser's addition theorem in groups,
{\it Proc. Amer. Math. Soc. } (1973), 443-451.

 \bibitem{dirac} G. A. Dirac, Extensions of Menger's theorem, {\it J. Lond. Math. Soc.}  38 (1963), 148–161.

\bibitem{dixmier}  J. Dixmier, Proof of a conjecture
by Erd\"os, Graham concerning the problem of Frobenius, {\it J.
number Theory} 34 (1990), 198-209.


\bibitem{dyson} F. J. Dyson,  A theorem on the densities of sets of integers.  {\it J. London Math. Soc}.  20,  (1945). 8--14.


\bibitem{erdheil}   P. Erd\H{o}s and H. Heilbronn,
{ On the Addition of residue classes mod  $p$}, {\it Acta Arith} 9
(1964), 149--159.

\bibitem{greenruz} B. Green, I. Z. Ruzsa,
 Sets with small sumset and rectification. {\it Bull. London Math. Soc.} 38 (2006), no. 1, 43--52.


\bibitem{davkem} D. Grynkiewicz, A step beyond Kempermann structure Theorem, Preprint May 2006.



\bibitem{hcras} Y.O. Hamidoune, Sur les atomes d'un graphe orient\'e,
 {\it C.R. Acad. Sc. Paris A}  284 (1977),   1253--1256.


\bibitem{hjct} Y.O. Hamidoune,   Quelques probl\`emes de connexit\'e dans les
graphes orient\'es,  {\it J. Comb. Theory} B 30 (1981), 1-10.


\bibitem{hejc1} Y.O. Hamidoune,   An application  of  connectivity theory in graphs to
      factorizations of elements in groups,  {\it Europ. J of Combinatorics} 2 (1981), 349-355.



\bibitem{hdiderich} Y.O. Hamidoune, On
a subgroup contained in words with a bounded length,
 {\it Discrete Math.} 103 (1992), 171--176.

\bibitem{hejc2} Y.O. Hamidoune, On the connectivity of Cayley digraphs,
 {\it Europ. J.  Combinatorics}, 5 (1984), 309-312.


\bibitem{hnetwork} Y.O. Hamidoune, Additive group theory applied to network topology. {\it Combinatorial network theory},  1--39, Appl. Optim., 1, Kluwer Acad. Publ., Dordrecht, 1996.
\bibitem{hejc3} Y.O. Hamidoune, On subsets with a small sum in abelian groups I:
The Vosper property, {\it Europ. J. of Combinatorics} 18 (1997),
541-556.

\bibitem{halgebra} Y.O. Hamidoune, An isoperimetric method in additive theory.
{\it J. Algebra} 179 (1996), no. 2, 622--630.


 \bibitem{hast} Y.O. Hamidoune,  On small subset product in a group.
Structure Theory of set-addition,  {\it Ast\'erisque}  no. 258(1999),
xiv-xv, 281--308.

 \bibitem{hactaa} Y.O. Hamidoune, {Some results in Additive number
Theory I: The critical pair Theory}, Acta Arith. 96, no. 2(2000),
97-119.

 \bibitem{hkneser} Y.O. Hamidoune, Y.O. Hamidoune,
  The global isoperimetric methodology applied to Kneser's Theorem, Preprint August-2007.


\bibitem{hkemp} Y.O. Hamidoune, Hyper-atoms and the Kemperman's critical pair Theory, Preprint, septembre 2007.

 \bibitem{hkempp} Y.O. Hamidoune, Beyond Kemperman's
Structure Theory: The isoperimetric approach, In preparation.



\bibitem{aoysc}  Y.O. Hamidoune, A.S. Ll\`ado, O. Serra, vosperian and
superconnected abelian Cayley digraphs,  {\it Graphs and Combinatorics} 7 (1991),
143--152.



\bibitem{aoytfree}  Y. O. Hamidoune, A. Ll\`ado, O. Serra,
{On subsets with a small product in torsion-free groups.}
{\it Combinatorica} 18(4) (1998), 513--549.


\bibitem{hrodseth1} Y. O. Hamidoune, {\O}. R{\o}dseth,   On bases for $\sigma$-finite groups.  Math. Sc,.  78  (1996),  no. 2, 246--254.

\bibitem{hrodseth2} Y. O. Hamidoune, {\O}. J. R{\o}dseth, An inverse theorem modulo $p$,
{\it Acta Arithmetica} 92 (2000)251--262.



\bibitem{hgoaa1}Y. O. Hamidoune, O. Serra, G. Z\'emor,
{On the critical pair theory in $\zp$},
 {\it Acta Arith.} 121 (2006), no. 2, 99--115.


\bibitem{hgochowla}Y. O. Hamidoune, O. Serra, G. Z\'emor,
{On the critical pair theory in Abelian groups: Beyond Chowla's Theorem}, Preprint september 2006.\\


\bibitem{hegyvary} N. Hegyv\'ari, NorbertOn iterated difference sets in groups,
{\it Period. Math. Hungar.} 43 (2001), no. 1-2, 105--110.




\bibitem{jin} R. Jin, Solution to the inverse problem for upper asymptotic density.  {\it J. Reine Angew. Math.}  595  (2006), 121--165.



\bibitem{rainbow} V. Jungi\'c, J. Licht, M. Mahdian, J. Ne\v set\v ril, R. Jaroslav, R. Radoi\v ci\'c,  Rainbow arithmetic progressions and anti-Ramsey results. Special issue on Ramsey theory. {\it Combin. Probab. Comput.} 12 (2003), no. 5-6, 599--620.

\bibitem{karolyi} G. K\'arolyi,
Cauchy-Davenport theorem in group extensions,
 {\it Enseign. Math.} (2) 51 (2005), no. 3-4, 239--254.

\bibitem{kempacta} J. H. B. Kemperman, On small sumsets in Abelian groups,
{\it Acta Math.} 103 (1960), 66--88.

\bibitem{kempcompl} J.H.B. Kempermann, On complexes in a semigroup,  {\it Nederl. Akad. Wetensch. Proc. Ser. A.} 59=  {\it Indag. Math.} 18(1956), 247--254.
 \bibitem{knesrdensite} M. Kneser,  Absch\"atzung der asymptotischen Dichte von Summenmengen. {\it  Math. Z.}  58,  (1953). 459--484.

\bibitem{knesrcomp} M. Kneser, Summenmengen in lokalkompakten  abelesche Gruppen,
{\it Math. Zeit.} 66 (1956), 88--110.




\bibitem{manlivre} H.B. Mann, {\it Addition Theorems},   R.E.
Krieger, New York, 1976.
\bibitem{manams} H. B. Mann,  An addition theorem for sets of elements of an Abelian group,{\it Proc. Amer. Math. Soc.} 4 (1953), 423.

 \bibitem{menger}  K.  Menger,  Zur allgemeinen Kurventhoerie. {\it Fund. Math}. 10 Karl (1927), 96-115.


\bibitem{natlivre}M. B. Nathanson,
{\it Additive Number Theory. Inverse problems and the geometry of
sumsets}, Grad. Texts in Math. 165, Springer, 1996.

\bibitem{segal1} N. Nikolov, D. Segal,  On finitely generated profinite groups. I. Strong completeness and uniform bounds. {\it Ann. of Math}. (2)  165  (2007),  no. 1, 171--238.

\bibitem{segal2} N. Nikolov, D. Segal, On finitely generated profinite groups. II. Products in quasisimple groups.  Ann. of Math. (2)  165  (2007),  no. 1, 239--273.

\bibitem{olsonaa} J.E. Olson, Sums of sets of group elements,
{\it Acta Arith.} 28 (1975/76), no. 2, 147--156.
\bibitem{olsonjnt} J.E. Olson, On the sum of two sets in a group,
  {\it J. Number Theory} 18 (1984), 110--120.


\bibitem{olsonsdif} J.E. Olson, On the symmetric difference of two sets in a group,
 {\it Europ. J. Combinatorics}, (1986), 43--54.





 \bibitem{klfree} A. Plagne, $(k,l)$-free sets in $\mathbb Z/p\mathbb Z$ are arithmetic progressions. {\it Bull. Austral. Math. Soc.} 65 (2002), no. 1, 137--144.
 \bibitem{pyber} L. Pyber,  Bounded generation and subgroup growth.  {\it Bull. London Math. Soc.}  34  (2002),  no. 1, 55--60.


\bibitem{rodsethfrob} {\O}. J. R{\o}dseth,  Two remarks on linear forms in non-negative
integers,  {\it  Math. Scand.} 51 (1982), 193--198.

\bibitem{ruz} I. Ruzsa,  An application of graph theory to additive number theory, {\it Scientia}, Ser A  3 (1989), 97–109.


\bibitem{serra} O. Serra,  An isoperimetric method for the small sumset problem. {\it Surveys in combinatorics} 2005, 119--152,
{\it London Math. Soc. Lecture Note Ser.}, 327, Cambridge Univ. Press, Cambridge, 2005.
\bibitem{SZ}  O. Serra, G. Z\'emor,
On a generalization of a theorem by Vosper,{\it Integers} 0, A10,
(electronic) 2000.\\

\bibitem{sheph} J. C. Shepherdson, On the addition of elements of a sequence,  {\it J. London Math Soc.} 22(1947), 85--88.

 \bibitem{szemeredi} E. Szemer\'edi, {\it Acta Arithmetica} 17 (1970),
 227--229,
\bibitem{tv} T. Tao, V.H. Vu,  {\it Additive Combinatorics}, Cambridge Studies
in Advanced Mathematics 105 (2006), Cambridge University Press.




\bibitem{vosper1} G. Vosper, The critical pairs of subsets of a group of prime order,
 {\it J. London Math. Soc.} 31 (1956), 200--205.

\bibitem{vosper2} G. Vosper, Addendum to "The critical pairs of subsets
of a group of prime order",
 {\it J. London Math. Soc.} 31 (1956), 280--282.



\bibitem{zemordam} G. Z\'emor, On positive and negative atoms of Cayley
digraphs,
{\it Discrete Appl. Math.} 23 (1989), no. 2, 193--195.





\end{thebibliography}
\end{document}